\documentclass[reqno]{amsart}

\usepackage{amsmath, amsfonts, amssymb, hyperref, tikz}
\usepackage{graphicx}
\usepackage{mathtools} 
\usepackage{faktor}
\usepackage{ytableau}
\usepackage{comment}
\usepackage[thinlines]{easytable}
\usepackage{enumerate}
\usepackage{adjustbox}
\usepackage{longtable}

\definecolor{light-gray}{gray}{0.8}

\usetikzlibrary{matrix,arrows,shapes}
\usetikzlibrary{backgrounds}

\usepackage[thinlines]{easytable}

{\theoremstyle{plain}

\newtheorem{theorem}{Theorem}[section]
\newtheorem{conjecture}[theorem]{Conjecture}
\newtheorem{proposition}[theorem]{Proposition}
\newtheorem{lemma}[theorem]{Lemma}
\newtheorem{corollary}[theorem]{Corollary}

}

{\theoremstyle{definition}
\newtheorem{definition}[theorem]{Definition}
\newtheorem*{definition*}{Definition}

\newtheorem{example}[theorem]{Example}
}

\newtheorem{innercustomgeneric}{\customgenericname}
\providecommand{\customgenericname}{}
\newcommand{\newcustomtheorem}[2]{%
  \newenvironment{#1}[1]
  {%
   \renewcommand\customgenericname{#2}%
   \renewcommand\theinnercustomgeneric{##1}%
   \innercustomgeneric
  }
  {\endinnercustomgeneric}
}

\newcustomtheorem{customthm}{Theorem}
\newcustomtheorem{customlemma}{Lemma}
\newcustomtheorem{customconj}{Conjecture}
\newcustomtheorem{customcor}{Corollary}

\numberwithin{equation}{section}
\numberwithin{figure}{section}

\newcommand{\Z}{\mathbb{Z}}
\newcommand{\C}{\mathbb{C}}
\newcommand{\Q}{\mathbb{Q}}

\newcommand{\abs}[1]{\left|#1\right|}
\newcommand{\ds}{\displaystyle}

\newcommand{\I}{\mathcal{I}}

\newcommand{\choos}[2]{\genfrac(){0pt}{0}{#1}{#2}}

\DeclareMathOperator{\len}{\ell}

\DeclareMathOperator{\sgn}{sgn}
\DeclareMathOperator{\Ce}{Z}

\begin{document}
\title{Hook-Shape Immanant Characters from Stanley-Stembridge Characters}

\author{Nathan R.\,T. Lesnevich}
\thanks{The author was partially supported by NSF grant DMS-1954001, and would like to thank Martha Precup and John Shareshian for their continued advice and support.}
\address{Department of Mathematics and Statistics, Washington University in St.\ Louis, One Brookings Drive, St.\ Louis, MO 63130, USA}
\email{\href{mailto:nlesnevich@wustl.edu}{nlesnevich@wustl.edu}}


\begin{abstract}
We consider the Schur-positivity of monomial immanants of Jacobi-Trudi matrices, in particular whether a non-negative coefficient of the trivial Schur function implies non-negative coefficients for other Schur functions in said immanants. We prove that this true for hook-shape Schur functions using combinatorial methods in a representation theory setting. Our main theorem proves that hook-shape immanant characters can be written as finite non-negative integer sums of Stanley-Stembridge characters, and provides an explicit combinatorial formula for these sums. This resolves a special case of a longstanding conjecture of Stanley and Stembridge that posits such a sum exists for all immanant characters. We also provide several simplifications for computing immanant characters, and several corollaries applying the main result to cases where the coefficient of the trivial Schur function in monomial immanants of Jacobi-Trudi matrices is known to be non-negative.
\end{abstract}

\maketitle

\tableofcontents
\section{Introduction}\label{sec:intro}
A fundamental class of objects in the theory and construction of symmetric functions is Jacobi-Trudi matrices. They are indexed by skew shapes, which are ordered pairs of partitions where the Young diagram of the second is contained in that of the first. Jacobi-Trudi matrices are of particular interest as their determinants are skew-Schur functions. In the case where the skew shape is simply a partition, the determinant is the Schur function indexed by that partition. Schur functions are essential in combinatorics and the representation theory of symmetric groups.\par 
Less studied is the theory of immanants. A virtual character of the symmetric group is a function from the symmetric group to the integers that is constant on conjugacy classes. In particular, characters of representations are virtual characters. The sign character of $S_n$ is $w \mapsto \sgn(w)$, and appears in the definition of a determinant of an $n \times n$ matrix $M = [m_{ij}]$,
\[\det(M) = \sum_{w \in S_n} \sgn(w)m_{1,w(1)},...,m_{n,w(n)}.\] 
Immanants are analogues of determinants in which the sign character is replaced with a virtual character of the symmetric group. When the chosen virtual character is the character of an irreducible representation, the corresponding immanant is called \emph{ordinary}. This paper is motivated by the study of immanants that use virtual characters corresponding to monomial symmetric functions under the Frobenius characteristic map, called \emph{monomial immanants}. \par 
Combinatorialists have studied immanants of Jacobi-Trudi matrices, in particular which immanants can be expanded non-negatively in the monomial or Schur bases of symmetric functions, as the skew-Shur functions can for determinants. For any symmetric function this property is referred to as being monomial-positive or Schur-positive, respectively.\par 
It was originally conjectured by Goulden and Jackson \cite{Goulden_Jackson92} and proven by Greene \cite{Greene92} that ordinary immanants of Jacobi-Trudi matrices are monomial-positive. It was conjectured by Stembridge \cite{Stem92} and proven by Haiman \cite{Haiman} that ordinary immanants of Jacobi-Trudi matrices are Schur-positive.\par 
We are here concerned with the following related conjecture of Stembridge.
\begin{conjecture}\label{conj:Stem_imm}\cite[Conj. 4.1]{Stem92}
Monomial immanants of Jacobi-Trudi matrices are Schur-positive.
\end{conjecture}
In \cite[4.1]{Stem92}, Stembridge defined a virtual character $\Gamma^\theta_{\mu/\nu}$, where $\theta \vdash N$ is a partition and $\mu/\nu$ is a skew shape with such that $N=\abs{\mu/\nu}$. The character $\Gamma^\theta_{\mu/\nu}$ is defined in detail in Section 2 below. We call $\Gamma_{\mu/\nu}^\theta$ the \emph{immanant character}, so named as it yields an equivalent formulation of Conjecture \ref{conj:Stem_imm}. 
\begin{customconj}{\ref{conj:Stem_imm}${}^\prime$}\label{conj:Stem_G}\cite[Conj. 4.1${}^\prime$]{Stem92}
$\Gamma^\theta_{\mu/\nu}$ is the character of a permutation representation of $S_n$ whose transitive components are each isomorphic to the action of $S_n$ on cosets of a Young subgroup.
\end{customconj}
Conjectures \ref{conj:Stem_imm} and \ref{conj:main_char} are also stated in \cite{Stan_Stem93} using the language of symmetric functions.\par 
Perhaps better known (and more often studied, as in \cite{Stan95, Gasharov96, Sharesh_Wachs16, GuayPaquet13, Abreu_Nigro20, harada2017cohomology, brosnan2018unit}, and many others) is the Stanley-Stembridge conjecture \cite{Stem92, Stan_Stem93}, which is Conjecture \ref{conj:Stem_G} in the particular case that $\theta = (N)$ (the original Stanley-Stembridge conjecture is a more general statement, but was reduced to this form in \cite{GuayPaquet13}). Because of this, we call $\Gamma^{(N)}_{\mu/\nu}$ the \emph{Stanley-Stembridge character}.\par  
The following conjecture of Stanley and Stembridge reduces Conjecture \ref{conj:Stem_G} to the Stanley-Stembridge conjecture.
\begin{conjecture}\label{conj:main_char}\cite[Conj. 5.1]{Stan_Stem93}
Every immanant character $\Gamma^\theta_{\mu/\nu}$ is a non-negative integral sum of Stanley-Stembridge characters.
\end{conjecture}

Conjectures \ref{conj:main_char} and \ref{conj:Stem_imm} are proven assuming the skew shape $\mu/\nu$ contains no $2\times 2$ box in its Young diagram \cite[\S 2]{Stan_Stem93}. This paper proves the case of Conjecture \ref{conj:main_char} when $\theta$ is a hook-shape partition $(N-k,1,...,1)$ and $\mu/\nu$ is arbitrary. Our main theorem is the following.
\begin{customthm}{A}
Let $\theta$ be a hook-shape partition and $\mu/\nu$ a skew shape. Then the immanant character ${\ds\Gamma_{\mu/\nu}^\theta}$ is a non-negative integer sum of Stanley-Stembridge characters.
\end{customthm}
Theorem A is Corollary \ref{cor:hook_partition} below, which gives an explicit combinatorial construction of the Stanley-Stembridge character summands. \par
We apply Theorem A to prove new cases of Conjecture \ref{conj:Stem_G}.
\begin{customcor}{B}
Let $\theta$ be a hook partition and $\mu/\nu$ a skew shape such that $\mu/\nu$ either:
\begin{itemize}
\item is pre-abelian, or
\item contains no $3\times 3$ box.
\end{itemize} 
Then Conjecture \ref{conj:Stem_G} holds for $\Gamma^\theta_{\mu/\nu}$.
\end{customcor}
Corollary B is a combination of Corollary \ref{cor:preabelian} and Corollary \ref{cor:small}.\par 
 We now give a brief overview of the contents of this paper. Section \ref{sec:prelim} gives necessary constructions and definitions for our proofs. Subsection 2.3 in particular contains background material on the connection to Hessenberg functions that is necessary to understand the proof of our main theorem. Section \ref{sec:simp} contains computational reductions whose proofs are relegated to Appendix A. These reduce Conjectures \ref{conj:Stem_imm} and \ref{conj:main_char} to a smaller class of skew shapes. Section 4 contains the proof of Theorem A, including an explicit decomposition of $\Gamma^\theta_{\mu/\nu}$ into Stanley-Stembridge characters when $\theta$ is a hook partition. This section also contains several corollaries of interest, including the results of Corollary B.
 \section{Characters, Immanants, and Hessenberg Functions}\label{sec:prelim}
A \emph{partition} $\lambda$ of length $\ell=:\len(\lambda)$ of a positive integer $n$ is a weakly decreasing sequence $(\lambda_1,...,\lambda_\ell)$ of positive integers such that $\sum_{i=1}^\ell \lambda_i = n$. If $\lambda$ is a partition of $n$ we write $\lambda \vdash n$.  The \emph{Young diagram} of $\lambda$ is a collection of upper-left justified boxes containing $\lambda_i$ boxes in row $i$. A \emph{standard Young tableau} is a filling of a Young diagram with distinct integers from $[n]$ that increases along rows and down columns. A \emph{semi-standard Young tableau} is a filling of those boxes with positive integers weakly increasing along rows and strictly increasing down columns. The \emph{content} of a semi-standard Young tableau is the list $c = (c_1,\ldots)$ such that the $c_i$ is the number of times $i$ appears in the tableau.\par 

Given a partition $\lambda \vdash n$ and any sequence of non-negative integers $c$ that sum to $n$, the \emph{Kostka number} $K_{\lambda,c}$ is the number of semi-standard Young tableaux with shape $\lambda$ and content $c$. The value $K_{\lambda,c}$ is unaffected by re-ordering the entries of $c$ or removing zeros. For example, if $c = (4,2,3,1)$ and $c' = (1,2,3,4)$ then $K_{\lambda,c} = K_{\lambda,c'}$ for all $\lambda \vdash 10$.
\begin{example}
Let $\theta = (6,1,1)$ and $c = (2,2,3,1)$. Then $K_{\theta,c} = 3$ and the semistandard Young tableaux of shape $\theta$ and content $c$ are
\[
\begin{ytableau}
1&1&2&3&3&4\\
2\\3
\end{ytableau}\;\;\;\;\;\;
\begin{ytableau}
1&1&2&3&3&3\\
2\\4
\end{ytableau}\;\;\;\;\;\;
\begin{ytableau}
1&1&2&2&3&3\\
3\\4
\end{ytableau}.
\]
Note that $K_{\theta,c} = \choos{3}{2}$ in this case.
\end{example}
The following lemma shows that Kostka numbers associated to hook partitions are particularly nice.
\begin{lemma}\label{lem:hook_kostka}
Let $\theta = (N-k,1,...,1)$ be a partition of $N$ and $c$ a content with $r$ nonzero entries. Then $K_{\theta,c} = \choos{r-1}{k}$.
\end{lemma}
\begin{proof}
We may assume that $c_1,...,c_r$ are the nonzero entries of $c$. Consider the $\choos{r-1}{k}$ size-$k$ subsets of $\{2,...,r\}$. Note that the top-left box in any semi-standard tableaux of shape $\theta$ and content $c$ must be $1$. Given a semistandard tableau of shape $\theta$ and content $c$, the set of entries in rows $2$ through $k+1$ determine a unique size-$k$ subset of $\{2,...,r\}$.  This correspondence defines a bijection.
\end{proof}
\par 
Given two partitions $\mu$ and $\nu$ such that $\len(\mu) \geq \len(\nu)$ and $\mu_i \geq \nu_i$ for all $i$, then we say $\nu < \mu$ and the pair of partitions is called a \emph{skew shape}, denoted $\mu/\nu$. The Young diagram of $\mu/\nu$ is the diagram of $\mu$ with the boxes of $\nu$ removed. If $\mu/\nu$ is a skew shape, its \emph{length} $\ell(\mu/\nu)$ is the largest index $i$ such that $\nu_i \lneq \mu_i$. If $\mu \vdash N_\mu$ and $\nu \vdash N_\nu$ then the \emph{size} of $\mu/\nu$ is $\abs{\mu/\nu} := N_\mu-N_\nu$. Standard and semistandard tableaux of skew shapes are defined with the same conditions on rows and columns as for partitions.

\subsection{Characters and symmetric functions} 
Let $S_n$ be the symmetric group on $n$ letters, and $C(w)$ the conjugacy class and $Z(w)$ the centralizer of $w$ in $S_n$. When the particular symmetric group is not clear from context, we will write $C_n(w)$ and $Z_n(w)$ to denote the conjugacy class and centralizer of $w$ in $S_n$. A \emph{virtual character} of $S_n$ is a function from $S_n$ to $\Z$ that is constant on conjugacy classes. As conjugacy classes are in bijection with partitions, virtual characters are also functions from the set $\{\lambda \vdash n\}$ of partitions of $n$ to $\Z$. We will denote the conjugacy class of $S_n$ associated to $\lambda\vdash n$ by $C(\lambda)$. A virtual character is a  \emph{character} if it arises as the character of a representation of $S_n$.\par 
\begin{example} The length $\ell(w)$ of a permutation $w \in S_n$ is the number of inversions of $w$. The sign character of $S_n$ is defined by $\sgn(w) = (-1)^{\ell(w)}$. A slightly more complicated example is the character that counts the number of fixed points of each permutation: $w \mapsto \abs{\{i \in [n] \mid w(i)=i\}}$. Both are virtual characters, and also happen to be characters of the sign and natural representations of $S_n$ respectively.
\end{example}
Symmetric functions (with coefficients in $\Z$) are formal power series in $\Z[x_1,\ldots]$ invariant under any permutation of the variables. The symmetric functions form a graded ring over $\Z$ denoted by $\Lambda$ with several important bases. Each of these bases is indexed by partitions of positive integers. The bases used herein are the monomial, homogeneous, and Schur symmetric functions, denoted by $\{m_\lambda\}$, $\{h_\lambda\}$, and $\{s_\lambda\}$, respectively. We also make use of the power-sum symmetric functions $\{p_\lambda\}$, which form a basis of $\Q \otimes \Lambda$. For more information on the enumerative combinatorics of symmetric function bases see \cite{Stanley_Vol1, Stanley_Vol2}.\par 
There is a natural inner product on the space of virtual characters.
If $\chi$ and $\psi$ are virtual characters of $S_n$, the \emph{character inner product} is the bilinear map on the space of virtual characters given by
\begin{align*}
\langle \chi, \psi \rangle &= \frac{1}{n!} \sum_{w \in S_n} \chi(w) \psi(w).
\end{align*}
Much like $\Lambda$, the space of all virtual characters of symmetric groups can also be given a graded ring structure. The \emph{induction product} of virtual characters $\phi$ of $S_k$ and $\psi$ of $S_\ell$ is
\[
\phi \circ \psi := \left.\left( \phi \times \psi \right)\right\uparrow_{S_k \times S_\ell}^{S_{k+\ell}},
\]
which is itself a character of $S_{k+\ell}$ \cite{sagan2013symmetric}.\par 
There is also an inner product on the ring of symmetric functions. Let $\langle \cdot,\cdot \rangle \colon \Lambda \time \Lambda \to \Q$ be defined so that the Schur symmetric functions basis is orthonormal. With this inner product the monomial and homogeneous functions form dual bases, so 
\[\langle m_\lambda,h_\mu \rangle = \begin{cases} 1 & \lambda = \mu \\ 0 & \text{otherwise}\end{cases}\].\par 
There is an isometric isomorphism between the ring of symmetric functions and the ring of virtual characters of all symmetric groups via the \emph{Frobenius characteristic map} $\mathrm{Frob}$, which produces a symmetric function from a virtual character $\chi$ on $S_n$ defined by:
\[
\mathrm{Frob}(\chi) := \sum_{\mu \vdash n} z_\mu^{-1} \chi(\mu) p_\mu, \;\;\;\text{where}\;\; z_\mu = \frac{n!}{\abs{C(\mu)}}.
\] 
The Frobenius characteristic map sends the characters of irreducible representations to the Schur symmetric function basis and the characters of representations defined by the action of the symmetric group on cosets of Young subgroups to the homogeneous symmetric function basis. The monomial symmetric functions are mapped to by virtual characters called \emph{monomial virtual characters}. We fix the notation 
\[
\mathrm{Frob}^{-1}(s_\lambda) =: \chi^\lambda,\;\; \mathrm{Frob}^{-1}(h_\lambda) =: \eta^\lambda,\; \text{ and }\; \mathrm{Frob}^{-1}(m_\lambda) =: \phi^\lambda.
\]
Since the $\chi^\lambda$ are characters of irreducible representations, we will call them \emph{irreducible characters}. Since the $\{\eta^\lambda\}$ correspond to induced characters of the trivial character on Young subgroups, we will call them \emph{induced trivial characters}. As such, if $\chi$ is a virtual character of $S_n$, the following are equivalent:
\begin{enumerate}
\item $\chi$ is the character of a permutation representation of $S_n$ whose transitive components are each isomorphic to the action of $S_n$ on cosets of a Young subgroup,
\item ${\ds \mathrm{Frob}(\chi) = \sum_{\lambda \vdash n} c_\lambda h_\lambda}$ where every $c_\lambda$ is a non-negative integer (i.e. $\mathrm{Frob}(\chi)$ is $h$-positive), and
\item ${\ds \chi = \sum_{\lambda \vdash n} c_\lambda \eta^\lambda}$ where every $c_\lambda$ is a non-negative integer.
\end{enumerate}
  Recall that $\{h_\lambda\}$ and $\{m_\lambda\}$ are dual bases in symmetric functions, so $\{\eta^\lambda\}$ and $\{\phi^\lambda\}$ are dual bases in the space of virtual characters of $S_n$. 

For more information on the correspondence between symmetric functions, characters, and representations see \cite{sagan2013symmetric}, and for a very thorough treatment of  symmetric functions see \cite{Stanley_Vol2}.\par 
\subsection{Immanants and the Immanant Character}
An immanant is a generalization of the determinant where the sign character is replaced with any virtual character. 
\begin{definition}
Let $M = [m_{ij}]_{1\leq i,j \leq n}$ be an $n\times n$ matrix with entries from an algebra over $\C$, and $\chi \colon S_n \to \Z$ a virtual character. The \emph{immanant} of $M$ with respect to $\chi$ is 
\[
\chi[M] := \sum_{w \in S_n} \chi(w) m_{1,w(1)}\cdots m_{n,w(n)}.
\]
When $\chi = \chi^\lambda$ is the character of an irreducible representation $\chi^\lambda[M]$ is referred to as an \emph{ordinary immanant}, and when $\chi = \phi^\lambda$ is a monomial virtual character $\phi^\lambda[M]$ is called a \emph{monomial immanant}.
\end{definition}
We consider matrices of symmetric functions, particularly Jacobi-Trudi matrices. 
\begin{definition}
Let $\mu/\nu$ be a skew shape of length $n$. The \emph{Jacobi-Trudi matrix} $H_{\mu/\nu}$ associated to $\mu/\nu$ is the $n \times n$ matrix whose $(i,j)$-th entry is the homogeneous symmetric function $h_{\mu_i-\nu_j+i-j}$,
\[
H_{\mu/\nu} := [h_{\mu_i-\nu_j+i-j}]_{1\leq i,j\leq n}.
\]
We set $h_0 = 1$, and if $\mu_i-\nu_j+i-j < 0$ then we set $h_{\mu_i-\nu_j+i-j}=0$.
\end{definition}
\begin{example}
The following are some (virtual) characters of $S_3$ and the Jacobi-Trudi matrix associated to skew shape $(2,2,2)/(1)$. Note that $\phi^{(2,1)}$ is the monomial character corresponding to $m_{(2,1)}$.
\begin{center}
\begin{tabular}{c|ccc}
Char         &$(1,1,1)$&$(2,1)$&$(3)$   \\\hline
$\sgn$         & $1$& $-1$& $1$      \\
$\chi^{(2,1)}$    & $2$& $0$& $-1$     \\
$\phi^{(2,1)}$   & $0$& $2$& $-3$
\end{tabular}
${\ds \;\;\;\;\;\;\;
H_{(2,2,2)/(1)} = 
\begin{bmatrix}
h_1 & h_3 & h_4 \\
1 & h_2 & h_3 \\
0 & h_1 & h_2 \\
\end{bmatrix}
}$
\end{center}
Computing the associated immanants, we obtain
\begin{align*}
\sgn\left[H_{222/1}\right] &= s_{(2,2,2)/(1)} = s_{(2,2,1)} \\ \\
\chi^{(2,1)}\left[H_{(2,2,2)/(1)}\right] &=2h_{(2,2,1)} - h_{(4,1)}\\
&= 2s_{(2,2,1)} + 2s_{(3,1,1)} + 4s_{(3,2)} + 3s_{(4,1)} + s_{(5)}\\ \\
\phi^{(2,1)}\left[H_{(2,2,2)/(1)}\right] &=2(h_{(3,2)}+h_{(3,1,1)}) - 3h_{(4,1)}\\
&= 2s_{(3,1,1)} + 4s_{(3,2)} + 3s_{(4,1)} + s_{(5)}.
\end{align*}
\end{example}
The determinants of $H_{\mu/\nu}$ are well studied, as by the Jacobi-Trudi identity $\det (H_{\mu/\nu}) = s_{\mu/\nu}$ is a skew-Schur function. In particular if $\nu = \emptyset$ then $\det (H_{\mu/\nu}) = s_\mu$ is a Schur function. Skew-Schur functions are known to be Schur positive, a fact that follows from the Littlewood-Richardson rule \cite{Stanley_Vol2}. Recall that Haiman proved 
\begin{theorem}\label{thm:ordinary_schur}\cite{Haiman}
Ordinary immanants of Jacobi-Trudi matrices are Schur-positive.
\end{theorem}
Conjecture \ref{conj:Stem_imm} considers whether monomial immanants of Jacobi-Trudi matrices are Schur positive, as is the case with the determinant and other ordinary immanants.\par 
To study Conjecture \ref{conj:Stem_imm}, we introduce a character originally defined by Stembridge \cite{Stem92}.
\begin{definition}
Let $\theta \vdash N$, and $\mu/\nu$ a skew shape with $\abs{\mu/\nu} = N$. Let $n$ be at least the length of $\mu/\nu$, and $w \in S_n$. Let $\delta := (n-1,...,1,0)$, and let $w \in S_n$ act on integer sequences by shuffling, so that $w(a_1,...,a_n) = (a_{w^{-1}(1)},...,a_{w^{-1}(n)})$. The \emph{immanant character} $\Gamma^\theta_{\mu/\nu}$ is the function
\begin{equation}\label{eqn:gamma}
\Gamma^{\theta}_{\mu/\nu}(w) = \frac{n!}{\abs{C(w)}} \sum_{w' \in C(w)} K_{\theta, \mu + \delta-w'(\nu+\delta)}
\end{equation}
We will denote $\mu + \delta-w'(\nu+\delta)$ as $\widehat{w'}$ when $\mu$, $\nu$, and $\delta$ are clear.
\end{definition}
The following is stated in \cite{Stem92} and we include the proof here for the reader's convenience.
\begin{lemma}\label{lem:inner_prod}\cite{Stem92}
Let $\phi$ be any virtual character. Then  the inner product $\left\langle\Gamma^\theta_{\mu/\nu},\phi\right\rangle$ is the coefficient of $s_\theta$ in the Schur expansion of the immanant $\phi[H_{\mu/\nu}]$.
\end{lemma}
\begin{proof}
First, the coefficient of $s_\theta$ in the Schur expansion of $h_\lambda$ is the Kostka number $K_{\theta,\lambda}$. Let $\phi$ be a virtual character. The character inner product of $\Gamma^\theta_{\mu/\nu}$ with $\phi$ is
\begin{align*}
\left\langle \Gamma^\theta_{\mu/\nu}, \phi  \right\rangle&= \frac{1}{n!} \sum_{w \in S_n} \Gamma^\theta_{\mu/\nu}(w) \phi(w)\\
&= \frac{1}{n!} \sum_{w \in S_n} \left(\frac{n!}{\abs{C(w)}} \sum_{w' \in C(w)} K_{\theta, \widehat{w'}}\right) \phi(w)\\
&= \sum_{w \in S_n} K_{\theta, \widehat{w}} \phi(w).
\end{align*}
On the other hand, the immanant of $H_{\mu/\nu}$ with respect to $\phi$ is
\begin{align*}
\phi[H_{\mu/\nu}] &= \sum_{w\in S_n} \phi(w) \prod_{i=1}^n h_{\mu_i - \nu_{w(i)} + w(i) - i} \\
&= \sum_{w\in S_n} \phi(w) h_{\widehat{w}} \\
&= \sum_{w\in S_n} \phi(w) \sum_{\theta \vdash N} K_{\theta,\widehat{w}} s_\theta\\
&= \sum_{\theta \vdash N} \left( \sum_{w\in S_n} \phi(w)  K_{\theta,\widehat{w}}\right)s_\theta.
\end{align*}
This concludes the proof.
\end{proof}
Recall that $\eta^\lambda$ and $\phi^\lambda$ are dual bases, so $\langle \Gamma^\theta_{\mu/\nu}, \phi^\lambda\rangle$ is the coefficient of $\eta^\lambda$ in $\Gamma^\theta_{\mu/\nu}$. Thus Lemma \ref{lem:inner_prod} is the connection between Conjectures \ref{conj:Stem_G} and \ref{conj:Stem_imm} above. \par 
Theorem \ref{thm:ordinary_schur} states that ordinary immanants of Jacobi-Trudi matrices are Schur-positive, so $\left\langle \Gamma^\theta_{\mu/\nu}, \chi^\lambda  \right\rangle$ is a non-negative integer for all partitions $\lambda$. The irreducible characters form an orthonormal basis of the space of characters of $S_n$, so $\Gamma^\theta_{\mu/\nu}$ is an integer sum of irreducible characters, and in particular, $\Gamma^\theta_{\mu/\nu}$ is indeed a character. 

\subsection{Hessenberg functions} 
If $\theta = (N)$, then all of the Kostka numbers in equation (\ref{eqn:gamma}) are either $0$ or $1$ depending on whether or not $\mu+\delta-w(\nu+\delta)$ has any negative entries. Thus the character $\Gamma^{(N)}_{\mu/\nu}$ depends only on the pattern of zeros in the Jacobi-Trudi matrix $H_{\mu/\nu}$. \par 
The pattern of nonzero entries in a Jacobi-Trudi matrix corresponds to a combinatorial object called a Hessenberg function.
\begin{definition}
A \emph{Hessenberg function} is a weakly increasing function $h \colon [n] \to [n]$ such that $h(i) \geq i$ for all $i \in [n]$. Each such function is denoted by a vector, $(h(1),h(2),...,h(n))$.
\end{definition}
The Hessenberg function $h$ determined by the pattern of zeros in  $H_{\mu/\nu}$ is given by  \[h(j) = \max\{i \in [n] \mid \mu_i-\nu_j + j-i \geq 0\}.\]
Recall that the $i$-th row and $j$-th column of $H_{\mu/\nu}$ is $h_{\mu_i-\nu_j + j-i}$. Thus in regards to the matrix $H_{\mu/\nu}$, $h(j)$ is the row index of the last nonzero entry in the $j$-th column of $H_{\mu/\nu}$.
\begin{example}\label{ex:hess_fcn} If $n = 5$, $\mu = (3,2,2,1,1)$ and $\nu = \emptyset$, then
\[
H_{\mu/\nu} = \begin{bmatrix}
h_3 & h_4&h_5&h_7&h_8\\
h_1 & h_2&h_3&h_4&h_5\\
1 & h_1&h_2&h_3&h_4\\
0 & 0&1&h_1&h_3\\
0 & 0&0&1&h_1\\
\end{bmatrix}.
\]
There are three nonzero entries in the first two columns, four in the third column, and five in the fourth and fifth columns. So the associated Hessenberg function is $h = (3,3,4,5,5)$.
\end{example} 
\begin{lemma}\label{lem:pos}
Let $\mu/\nu$ be a skew shape. Then $\widehat{w^{-1}}$ has no negative entries if and only if $w(i) \leq h(i)$ for all $i \in [n]$.
\end{lemma}
\begin{proof}
Fix $i \in [n]$. Then 
\[\left(\widehat{w^{-1}}\right)_{w(i)} = (\mu + \delta - w^{-1}(\nu + \delta))_{w(i)} = \mu_{w(i)} - \nu_{i} + i - w(i).\]
By definition $h(i) = \max\{j \in [n] \mid \mu_j-\nu_i + i-j \geq 0\}$. So if $w(i) > h(i)$ then $(\mu + \delta - w^{-1}(\nu + \delta))_{w(i)} < 0$. Similarly, if $w(i) \leq h(i)$ then $(\mu + \delta - w^{-1}(\nu + \delta))_{w(i)} \geq 0$.\par 
Since $\{w(i) \mid i \in [n]\} = [n]$, this concludes the proof.
\end{proof} 
In the case of Example \ref{ex:hess_fcn}, the set of $w \in S_n$ such that $\widehat{w^{-1}}$ has no negative entries is $\{w \in S_5 \mid w(1),w(2) \leq 3 \;\text{ and, }\; w(3) \leq 4\}$. Given a Hessenberg function $h$, the indicator function $\widehat{h} \colon S_n \to \{0,1\}$ will denote whether or not $\mu+\delta-w^{-1}(\nu+\delta)$ has negative entries. By Lemma \ref{lem:pos}, 
\[
\widehat{h}(w) := \begin{cases} 1 & w(i) \leq h(i) \text{ for all } i \in [n] \\
 0 & \text{otherwise}.\end{cases}
\]

\begin{example}\label{ex:indicator}
Let $h=(3,3,4,4)$. The following diagrams depict each permutation matrix imposed over a diagrams with box $(i,j)$ shaded whenever $i \leq h(j)$.
\begin{center}
\begin{tabular}{ccc}
$1234$ & $3142$ & $3412$\\
$
\begin{ytableau}
*(blue!30)1 &*(light-gray)  &*(light-gray) &*(light-gray) \\
*(light-gray) &*(blue!30)1 &*(light-gray) &*(light-gray) \\
*(light-gray) &*(light-gray) &*(blue!30)1 &*(light-gray) \\
 & &*(light-gray) &*(blue!30)1
\end{ytableau}$
&
$
\begin{ytableau}
*(light-gray) &*(blue!30)1  &*(light-gray) &*(light-gray) \\
*(light-gray) &*(light-gray) &*(light-gray) &*(blue!30)1 \\
*(blue!30)1 &*(light-gray) &*(light-gray) &*(light-gray) \\
 & &*(blue!30)1 &*(light-gray)
\end{ytableau}$
&$
\begin{ytableau}
*(light-gray) &*(light-gray)  &*(blue!30)1 &*(light-gray) \\
*(light-gray) &*(light-gray) &*(light-gray) &*(blue!30)1 \\
*(blue!30)1 &*(light-gray) &*(light-gray) &*(light-gray) \\
 &1 &*(light-gray) &*(light-gray)
\end{ytableau}$
\end{tabular}
\end{center}
According to the pictures, we get $\widehat{h}(1234) = 1$, $\widehat{h}(3142) = 1$, and $\widehat{h}(3412) = 0$.
\end{example}
 
Lemma \ref{lem:kostka_hess} below allows us to compute $\Gamma^{(N)}_{\mu/\nu}$ using only the data of the Hessenberg function corresponding to $\mu/\nu$.
\begin{lemma}\label{lem:kostka_hess}
Let $\mu/\nu$ be a skew shape with corresponding Hessenberg function $h$. Then 
\[\Gamma^{(N)}_{\mu/\nu}(w) = \frac{n!}{\abs{C(w)}} \sum_{w' \in C(w)} \widehat{h}(w')\]

\end{lemma}
\begin{proof}
The Kostka number $K_{(N),\widehat{w}}$ is $1$ whenever $\widehat{w}$ has only non-negative entries and $0$ otherwise. By Lemma \ref{lem:pos} this occurs precisely when $w^{-1}(i) \leq h(i)$ for all $i \in [n]$. Since $C(w)$ is closed under inverses, we may ignore the nuance of distinguishing $w$ and $w^{-1}$ in the sum for $\Gamma_h$.
\end{proof}
By Lemma \ref{lem:kostka_hess}, if $\mu/\nu$ is a skew shape with corresponding Hessenberg function $h$, we let $\Gamma_{\mu/\nu}^{(N)} = \Gamma_h$, where 
\[
\Gamma_h := \frac{n!}{\abs{C(w)}} \sum_{w' \in C(w)} \widehat{h}(w').
\]
The following example demonstrates how to compute a particular $\Gamma_h$.
\begin{example}
Let $h = (3,3,4,4)$. The following table lists each conjugacy class by the associated partition $\lambda$, the permutations $w$ in that conjugacy class such that $\widehat{h}(w) = 1$, and the value that $\Gamma_h$ takes on that conjugacy class.
\begin{center}
\begin{tabular}{c|c|c}
$\lambda$ & $\widehat{h}(w)=1$ & $\Gamma_h$ \\\hline
$(1,1,1,1)$ & $1234$ & $\frac{4!}{1} \cdot 1$  \\\hline
$(2,1,1)$ & $2134$ &  $\frac{4!}{6}\cdot 4$ \\
&$1324$&\\
&$1243$&\\
&$3214$\\\hline
$(2,2)$ & $2143$ & $\frac{4!}{3}\cdot 1$ \\\hline
$(3,1)$ & $2314$ & $\frac{4!}{8}\cdot 3$ \\
&$1342$&\\
&$3124$&\\\hline
$(4)$ & $2341$ & $\frac{4!}{6}\cdot 3$  \\
&$3142$&\\
&$3241$&
\end{tabular}
\end{center}
For example, the cycle type of the permutation $4231$ is $(2,1,1)$, so $\Gamma_h(4231) = \frac{4!}{6}\cdot 4 = 16$.
\end{example}

\section{Computational Simplifications}\label{sec:simp}
In this section, we give several simplifications for computing $\Gamma^{\theta}_{\mu/\nu}$. The methods of proof are technical representation-theoretic computations and are not used in the other sections of our paper. As such, the proofs are delayed to Appendix A.\par 
\ytableausetup{smalltableaux}
Let $\mu = (5,4,2,2,1)$ and $\nu = (3,2,2)$, and $\widehat{\mu} = (5,4,2,1)$ and $\widehat{\nu} = (3,2)$. The associated diagrams are as follows.
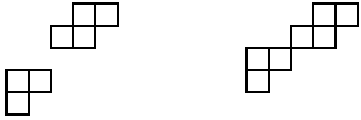
\begin{center}
\begin{tabular}{ccc}
$\ydiagram{3+2,2+2,0,2,1}$&\hspace{1cm} & $\ydiagram{3+2,2+2,2,1}$
\end{tabular}
\end{center}
Note that the skew shape  skew shape $\widehat{\mu}/\widehat{\mu}$ is simply the skew shape $\mu/\nu$ with the empty third row removed.\par
Proposition \ref{prop:no_zeros} below asserts that the immanant characters $\Gamma^\theta_{\mu/\nu}$ and $\Gamma^\theta_{\widehat{\mu}/\widehat{\nu}}$ are equal. In particular, it tells us that to compute $\Gamma_{\mu/\nu}^\theta$ it suffices consider skew shapes without empty (zero) rows in the middle.
\begin{proposition}\label{prop:no_zeros}
Suppose $\mu/\nu$ is a skew shape such that $\mu_i = \nu_i$ for some $i \in [n]$, where $n \geq \ell(\mu/\nu)$. Let $\widehat{\mu}$ and $\widehat{\nu}$ denote, respectively, the partitions $\mu$ and $\nu$ with their $i$-th components removed. Then $\Gamma^\theta_{\mu/\nu}(w) = \Gamma^\theta_{\widehat{\mu}/\widehat{\nu}}(w)$
for all $w \in S_n$ and $\theta \vdash N$.
\end{proposition}
Appending empty rows to a skew shape $\mu/\nu$ allows one to consider $\Gamma^\theta_{\mu/\nu}$ as a character of a symmetric group on more letters than $\mu/\nu$ has nonzero rows. Proposition \ref{prop:min_n} below confirms that this process does not meaningfully alter the character.
\begin{proposition}\label{prop:min_n}
Let $\mu/\nu$ be a skew shape of length at most $n-1$. If $\Gamma_{\mu/\nu}^\theta = \sum_i \Gamma_{\mu_i/\nu_i}^{(N)}$ as characters in $S_{n-1}$, then $\Gamma_{\mu/\nu}^\theta = \sum_i \Gamma_{\mu_i/\nu_i}^{(N)}$ in $S_{n}$. In particular, if Conjecture \ref{conj:main_char} is true for characters $\Gamma^{\theta}_{\mu/\nu}$ of $S_{\ell(\mu/\nu)}$ then it is true in general.
\end{proposition}
By Proposition \ref{prop:min_n}, we may always assume that $\Gamma_{\mu/\nu}^\theta$ is a character of $S_n$ where $n$ is the number of nonempty rows in the skew diagram.\par 
Consider the skew shapes $\mu^0/\nu^0 = (5,4,2,1)/(3,2)$ and $\mu^1/\nu^1 = (5,4,3,2)/(3,3,1)$. The respective diagrams are as follows.
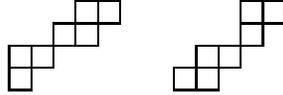
\begin{center}
\begin{tabular}{ccc}
$\ydiagram{3+2,2+2,2,1}$ & \hspace{1cm}&  $\ydiagram{3+2,3+1,1+2,2}$
\end{tabular}
\end{center}
From these diagrams, we see that we may swapped the order of the connected components of $\mu^0/\nu^0$ to obtain $\mu^1/\nu^1$. Proposition \ref{prop:connected_order} below asserts that the immanant characters $\Gamma^\theta_{\mu^0/\nu^0}$ and  $\Gamma^\theta_{\mu^1/\nu^1}$ are equal.
\begin{proposition}\label{prop:connected_order}
Suppose $\mu^0/\nu^0$ and $\mu^1/\nu^1$ are skew shapes whose skew diagrams have identical connected components. Then $\Gamma^\theta_{\mu^0/\nu^0} = \Gamma^\theta_{\mu^1/\nu^1}$.
\end{proposition}
Consider $\mu/\nu = (5,4,2,1)/(3,2)$, whose diagram appears above. The connected components of $\mu/\nu$ are $(3,2)/(1)$ and $(2,1)$. The following proposition allows one to compute the $\mu/\nu$ Stanley-Stembridge character of $S_4$ from the $(3,2)/(1)$ and $(2,1)$ Stanley-Stembridge characters both of $S_2$, and is due to Stanley and Stembridge. 
\begin{proposition}\label{prop:(N)_products}\cite[\S 5]{Stan_Stem93}
Let $\mu/\nu$ be a disconnected skew shape with components $\mu^0/\nu^0$ and $\mu^1/\nu^1$. Let $N_i=\abs{\mu^i/\nu^i}$. Then
\[
\Gamma^{(N)}_{\mu/\nu} = \Gamma^{(N_0)}_{\mu^0/\nu^0} \circ \Gamma^{(N_1)}_{\mu^1/\nu^1} .
\]
\end{proposition}
Proposition \ref{prop:(N)_products} can be generalized to all immanant characters. For example Proposition \ref{prop:disconnect} below allows one to compute the immanant character $\Gamma_{\mu/\nu}^\theta$ of $S_4$ using immanant characters of the skew shapes $(3,2)/(1)$ and $(2,1)$. In general, it allows one to compute $\Gamma_{\mu/\nu}^{\theta}$ from the immanant characters of the connected components of $\mu/\nu$.
\begin{proposition}\label{prop:disconnect}
Let $\mu/\nu$ be a disconnected skew shape with components $\mu^0/\nu^0$ and $\mu^1/\nu^1$. Let $N_i=\abs{\mu^i/\nu^i}$. Then
\[
\Gamma^{\theta}_{\mu/\nu} =\sum_{\substack{ \lambda \vdash N_0\\\lambda <\theta}} \sum_{\sigma\vdash N_1} c^\theta_{\lambda\sigma}\Gamma^{\lambda}_{\mu^0/\nu^0} \circ \Gamma^{\sigma}_{\mu^1/\nu^1},
\]
Where $c^{\theta}_{\lambda\sigma}$ is the Littlewood-Richardson coefficient.
\end{proposition}
We summarize the consequences of the above computational reductions to Conjectures \ref{conj:Stem_imm} and \ref{conj:main_char} in Corollary \ref{cor:comp_reduct} below.
\begin{corollary}\label{cor:comp_reduct}
If Conjectures \ref{conj:Stem_imm} and \ref{conj:main_char} hold for skew shapes $\mu/\nu$ and $n$ such that $\mu/\nu$ is connected and $n$ is the length of $\mu/\nu$, then they hold in full generality.
\end{corollary}
\begin{proof}
Assume the conjectures hold as in the claim. Proposition \ref{prop:min_n} ensures we may take $n$ to be the length of $\mu/\nu$. The immanant character for a disconnected skew shape can be written as the non-negative integral sum of immanant characters of connected components via Proposition \ref{prop:disconnect}. The induction product of induced trivial characters is itself an induced trivial character, so Conjecture \ref{conj:Stem_imm} follows from Proposition \ref{prop:disconnect}. The induction product distributes over sums of characters, so Conjecture \ref{conj:main_char} follows from Proposition \ref{prop:disconnect} and Proposition \ref{prop:(N)_products}.
\end{proof}
\ytableausetup{nosmalltableaux}

\section{The Hook Partition Case}
We aim to prove Conjecture \ref{conj:main_char} when $\theta$ is a hook  partition, which asserts that $\Gamma^{\theta}_{\mu/\nu}= \sum_{i \in I} \Gamma^{(N_i)}_{\mu^i/\nu^i}$, where the sum is over some finite index set $I$. Considering the value at $w = \mathrm{id}$, if the conjecture holds then $\abs{I} = K_{\theta,\mu-\nu}$. \par 
Lemma \ref{lem:suffices} below yields an avenue for a combinatorial proof for special cases of Conjecture \ref{conj:main_char}.
\begin{lemma}\label{lem:suffices}
Fix $\theta$ and $\mu/\nu$ and set $n = \ell(\mu/\nu)$. If there exists a finite set of Hessenberg functions $\{h^i \mid i \in I\}$ such that for all $w \in S_n$, $\widehat{h^i}(w)=1$ for precisely $K_{\theta,\widehat{w}}$-many $i \in I$, then Conjecture \ref{conj:main_char} holds for the character $\Gamma^{\theta}_{\mu/\nu}$. 
\end{lemma}
\begin{proof}
Say that 
\[
 K_{\theta, \mu + \delta-w'(\nu+\delta)} = \sum_{i \in I} \widehat{h^i}(w')
\]
for all $w' \in S_n$. Let $\lambda \vdash n$. Then
\begin{align*}
\frac{n!}{\abs{C(\lambda)}}\sum_{w' \in C(\lambda)} \left(K_{\theta, \mu + \delta-w'(\nu+\delta)} \right)&= \frac{n!}{\abs{C(\lambda)}}\sum_{w' \in C(\lambda)}\left(\sum_{i \in I} \widehat{h^i}(w')\right)\\
 &= \sum_{i \in I} \left( \frac{n!}{\abs{C(\lambda)}}\sum_{w' \in C(\lambda)} \widehat{h^i}(w')\right).
\end{align*}

Since $\lambda$ was arbitrary, ${\ds \Gamma^{\theta}_{\mu/\nu}= \sum_{i \in I} \Gamma_{h^i}}$. By Lemma \ref{lem:kostka_hess} the claim follows.
\end{proof}
Recall from Lemma \ref{lem:hook_kostka} that the Kostka numbers for hook partitions are particularly nice, as if $\ell(c)$ is the number of nonzero entries in $c$, and $\theta$ is a hook partition of length $k+1$, then ${\ds K_{\theta,c} = \choos{\ell(c)-1}{k}}$. \par 
The following allows us to apply Lemma \ref{lem:suffices} in the case where $\theta$ is a hook. This is the key combinatorial result of this section.
\begin{theorem}\label{thm:hook_partition}
Let ${\ds \theta = (N-k,1,...,1)}$ be a hook partition and $\mu/\nu$ a skew shape with no empty rows and with associated Hessenberg function $h \colon [n] \to [n]$. Let $\theta \vdash N$ and $\abs{\mu/\nu} = N$. Then 
\[
K_{\theta,\widehat{w}} = \sum_{\substack{J \subset [n-1]\\ \abs{J} = k}} \widehat{h^J}(w),
\]
for all $w \in S_n$, where for each $J \subset [n-1]$, $h^J$ is the Hessenberg function
\[
h^J(i) = \begin{cases} h(i) - 1 & i \in J \text{ and } \mu_{h(i)}-\nu_{i}+i-h(i) = 0\\
h(i) & \text{otherwise.}
\end{cases}
\]
\end{theorem}
\begin{proof}[Proof of Theorem \ref{thm:hook_partition}]
First we verify that $h^J$ is in fact a Hessenberg function. If $h^J(i) < i$, then since $h(i)-1 \leq h^J(i)$ either $h(i) < i$ or $h(i) = i$. If $h(i) < i$ then we reach contradiction as $h$ is a Hessenberg function. If $h(i) = i$ and $h^J(i) < i$ then $i \in J$ and $\mu_{h(i)}-\nu_i+i-h(i) = \mu_i-\nu_i+i-i = 0$. This contradicts our assumption that $\mu/\nu$ had no nonzero rows. So we have that $h^J(i) \geq i$ for all $i \in [n]$.\par 
Now we check that $h^J$ is non-decreasing. Since $h$ is non-decreasing, if $h^J(i) > h^J(i+1)$ then $h(i) = h(i+1)$ and $h(i+1)-1 = h^J(i+1)$. We have however that 
\begin{align*}
0 &\leq \mu_{h(i)}-\nu_{i}+i-h(i) \\&\lneq \mu_{h(i)} - \nu_{i+1}+(i+1)-h(i) \\
&= \mu_{h(i+1)} - \nu_{i+1}+(i+1)-h(i+1),
\end{align*}
so $\mu_{h(i+1)} - \nu_{i+1}+(i+1)-h(i+1) \neq 0$, and thus $h^J(i+1) = h(i+1) \geq h(i) \geq h^J(i)$. Thus every $h^J$ is a Hessenberg function. \par 
Since $h^J(i) \leq h(i)$ for all subsets $J$ and for all $i \in [n]$, $\widehat{h}(w) = 0$ implies that $\widehat{h^J}(w) = 0$ for all $J$. Thus it suffices to restrict our attention to those $w$ such that $\widehat{h}(w) = 1$.\par 
Recall $\widehat{w}$ denotes the sequence $\mu + \delta - w(\nu+\delta)$ of $n$ integers. Let $Z_w = \{i \in [n] \mid \widehat{w}_{w(i)} = 0\}$, and let $z_w = \abs{Z_w}$. Note $n \notin Z_w$. The content $\widehat{w}$ has $n-z_w$ many nonzero terms, and thus $K_{\theta,\widehat{w}} = \choos{n-z_w-1}{k}$ by Lemma \ref{lem:hook_kostka}.\par 
Consider $w \in S_n$ such that $\widehat{h}(w) = 1$ and $\widehat{h^J}(w) = 0$. We will show that $J \cap Z_w = \{i \in [n-1] \mid h^J(i) < w(i) = h(i)\}$. Say that $i \in J\cap Z_w$. Since $i \in Z_w$, $\mu_{w(i)}-\nu_{i}+i-w(i)= 0$, and so $w(i) = h(i)$. Since $i \in J$ and  $\mu_{h(i)}-\nu_{i}+i-h(i)=\mu_{w(i)}-\nu_{i}+i-w(i)=0$, we see that $h^J(i) = h(i)-1 = w(i)-1$. So $h^J(i) < w(i)$, and so $J \cap Z_w \subseteq \{i \in [n-1] \mid h^J(i) < w(i) = h(i)\}$. \par 
Let $i \in \{i \in [n-1] \mid h^J(i) < w(i) = h(i)\}$. Then $h^J(i) < h(i)$. By the construction of $h^J$, $i \in J$ and $\mu_{h(i)}-\nu_{i}+i-h(i)=0$. Since $w(i) = h(i)$, $\mu_{w(i)}-\nu_{i}+i-w(i) = 0$ and so $i \in Z_w$. Thus we have the other direction of containment and $J \cap Z_w = \{i \in [n-1] \mid h^J(i) < w(i) =h(i)\}$. \par 
We next claim that $\widehat{h^J}(w) = 1$ if and only if $J \cap Z_w = \emptyset$. We use the presentation $J \cap Z_w = \{i \in [n-1] \mid h^J(i) < w(i) =h(i)\}$. If $\widehat{h^J}(w) = 1$ then there exist no $i \in [n]$ such that $h^J(i) < w(i)$, and so $J \cap Z_w = \emptyset$. On the other hand, if $\widehat{h^J}(w) = 0$ then there exists an $i \in [n-1]$ such that $h^J(i) < w(i)$. Since $w(i) \leq h(i)$, it must be that $w(i) = h(i)$. So $i \in J\cap Z_w$ and $J \cap Z_w \neq \emptyset$. \par 
Now $J \cap Z_w = \emptyset$ exactly when $J \subset \left([n-1] \setminus Z_w\right)$. There are precisely \[\ds \choos{\abs{[n-1]\setminus Z_w}}{\abs{J}} = \choos{n-z_w-1}{k} = K_{\theta,\widehat{w}} \] many such subsets $J$. This concludes the proof.
\end{proof}
We apply Theorem \ref{thm:hook_partition} to obtain an expansion for hook partition immanant characters in terms of Stanley-Stembridge characters.
\begin{corollary}\label{cor:hook_partition}
Let ${\ds \theta = (N-k,1,...,1)}$ be a hook partition and $\mu/\nu$ a skew shape with no empty rows and associated Hessenberg function $h \colon [n] \to [n]$. Let $\theta \vdash N = \abs{\mu/\nu}$. Then 
\begin{equation}\label{eqn:expansion}
\Gamma^{\theta}_{\mu/\nu} = \sum_{\substack{J \subset [n-1]\\ \abs{J} = k}} \Gamma_{h^J},
\end{equation}
where 
\[
h^J(i) = \begin{cases} h(i) - 1 & i \in J \text{ and } \mu_{h(i)}-\nu_{i}+i-h(i) = 0\\
h(i) & \text{otherwise.}
\end{cases}
\]
Furthermore, if we collect terms in \emph{(\ref{eqn:expansion})} so that 
\[
\Gamma^{\theta}_{\mu/\nu} = \sum_{J} c_J \Gamma_{h^J}
\]
where each $h^J$ is a unique Hessenberg function, then 
\[c_J = \choos{a}{b}, \text{  where  } \begin{tabular}{l}$a = \abs{\{i \in [n-1] \mid \mu_{h(i)}-\nu_{i}+i-h(i) > 0\}}$\\[6pt]$b=k - \abs{\{ i \in [n-1] \mid h(i) \neq h^J(i)\}}.$\end{tabular}\]
\end{corollary}
\begin{proof}
Equation (\ref{eqn:expansion}) follows directly from Lemma \ref{lem:suffices} and Theorem \ref{thm:hook_partition}. By construction $c_J$ is the number of $J' \subset [n]$ such that $\abs{J'} = k$ and $h^J = h^{J'}$. Those $J'$ must contain the $i \in [n]$ such that $i \in J$ and $h^J(i) < h(i)$. The remaining $k - \abs{\{ i \in [n-1] \mid h(i) \neq h^J(i)\}}$ elements of $j \in J'$ can be any $j$ such that $h^{J'}(j) = h(j)$, in particular any $j \in [n-1]$ such that $\mu_{h(j)}-\nu_{j}+j-h(j) > 0$.
\end{proof} 
Given Lemma \ref{lem:inner_prod}, Corollary \ref{cor:hook_partition} states that in the Schur expansion of the $\psi$-immanant of a Jacobi-Trudi matrix for any virtual character $\psi$, the hook partition Schur coefficients are non-negative sums of trivial Schur coefficients in the Schur expansions for $\psi$-immanants of some collection of Jacobi-Trudi matrices.
\begin{example} 
Let $\theta = (6,1,1)$ and $\mu/\nu = (3,3,3,1)/(1,1)$ so $h = (3,3,4,4)$. The subsets $J \subset [3]$ and Hessenberg functions $h^J$ from Theorem \ref{thm:hook_partition} are\vspace{0.5cm}
\begin{center}
\begin{tabular}{r|c|l}
$J \subset \{1,2,3\}$ &  & $h^J$ \\\hline
$\{1,2\}$ & \parbox{1.8in}{$\mu_{h(1)}-\nu_1+1-h(1) =0$ \\ $\mu_{h(2)}-\nu_2+2-h(2)=1$} & $(2,3,4,4)$ \\\hline
$\{1,3\}$ &\parbox{1.8in}{$\mu_{h(1)}-\nu_1+1-h(1) =0$ \\ $\mu_{h(3)}-\nu_3+3-h(3)=0$}  & $(2,3,3,4)$ \\\hline
$\{2,3\}$ &\parbox{1.8in}{$\mu_{h(2)}-\nu_2+2-h(2) =1$ \\ $\mu_{h(3)}-\nu_3+3-h(3)=0$}  & $(3,3,3,4)$
\end{tabular}
\end{center}
\vspace{0.5cm}
and so 
\[K_{(6,1,1),\widehat{w}} = h^{\{1,2\}}(w) + h^{\{1,3\}}(w) + h^{\{2,3\}}(w)\]
for all $w \in S_n$, and in particular
\[
\Gamma^{(6,1,1)}_{(3,3,3,1)/(1,1)} = \Gamma_{(2,3,4,4)}+\Gamma_{(2,3,3,4)}+\Gamma_{(3,3,3,4)}.
\]
We may also visualize the Hessenberg function $h^J$ for each subset $J$ as follows. Look at the corners of the Hessenberg function cut out in the Jacobi-Trudi matrix and remove the corner if it contains a $1$ and the column is indexed by an element of $J$.
\begin{center}
\begin{tabular}{ccc}
$\{1,2\}$ & $\{1,3\}$ & $\{2,3\}$ \\
${\ds 
\begin{ytableau}
\none[\downarrow] & \none[\downarrow] & \none & \none \\
*(light-gray)h_2 &*(light-gray) h_3 &*(light-gray) h_5 & *(light-gray)h_6\\
*(light-gray)h_1 &*(light-gray) h_2 & *(light-gray)h_4 & *(light-gray)h_5\\
*(pink)1   &*(light-gray) h_1   &*(light-gray) h_3 & *(light-gray)h_4\\
0 & 0 & *(light-gray)1    & *(light-gray)h_1\\
\none[2] & \none[3] & \none[4] & \none[4]
\end{ytableau}
}$
&
${\ds 
\begin{ytableau}
\none[\downarrow] & \none & \none[\downarrow] & \none \\
*(light-gray)h_2 &*(light-gray) h_3 &*(light-gray) h_5 & *(light-gray)h_6\\
*(light-gray)h_1 &*(light-gray) h_2 & *(light-gray)h_4 & *(light-gray)h_5\\
*(pink)1   &*(light-gray) h_1   &*(light-gray) h_3 & *(light-gray)h_4\\
0 & 0 & *(pink)1    & *(light-gray)h_1\\
\none[2] & \none[3] & \none[3] & \none[4]
\end{ytableau}
}$
&
${\ds 
\begin{ytableau}
\none & \none[\downarrow] & \none[\downarrow] & \none \\
*(light-gray)h_2 &*(light-gray) h_3 &*(light-gray) h_5 & *(light-gray)h_6\\
*(light-gray)h_1 &*(light-gray) h_2 & *(light-gray)h_4 & *(light-gray)h_5\\
*(light-gray)1   &*(light-gray) h_1   &*(light-gray) h_3 & *(light-gray)h_4\\
0 & 0 & *(pink)1    & *(light-gray)h_1\\
\none[3] & \none[3] & \none[3] & \none[4]
\end{ytableau}
}$\\

\end{tabular}
\end{center}
The Hessenberg functions $h^{\{1,2\}} = (2,3,4,4)$, $h^{\{1,3\}} = (2,3,3,4)$, and $h^{\{2,3\}} = (3,3,3,4)$ are easily obtained from the above diagrams.
\end{example}
As an application of our result, we apply Theorem \ref{thm:hook_partition} where the Stanley-Stembridge conjecture is already known in order to prove the hook partition version of Conjecture \ref{conj:Stem_imm} in those cases. 
 A Hessenberg function is \emph{abelian} if $h(h(1)+1)) = n$ or if $h(1)=n$. In the abelian case, the Stanley-Stembridge Conjecture is known.
\begin{theorem}\label{thm:abelian}\cite{harada2017cohomology}
If $h$ is abelian, then $\Gamma_h$ is the character of a permutation representation of $S_n$ whose transitive components are each isomorphic to the action of $S_n$ on cosets of a Young subgroup.
\end{theorem}
\begin{definition}\label{def:preabelian}
Let $\mu/\nu$ be a skew shape and $H_{\mu/\nu}$ the associated Jacobi-Trudi matrix. Let $H_{\mu/\nu}'$ be the matrix obtained from $H_{\mu/\nu}$ by replacing all $1$-s with $0$-s. The pattern of nonzero entries in $H_{\mu/\nu}'$ determines a Hessenberg function $h'$. We say the skew shape $\mu/\nu$ is \emph{pre-abelian} if $h'$ is an abelian Hessenberg function.
\end{definition}
\begin{example}
The skew shapes $(3,3,3,1)/(1,1)$, $(4,3,2,1)$, and $(2,2,1,1)/(1,1)$ for example are not pre-abelian. If we were to replace ``$1$"-s with ``$0$"-s in the Jacobi-Trudi matrices $H_{(3,3,3,1)/(1,1)}$, $H_{(4,3,2,1)}$, and $H_{(2,2,1,1)/(1,1)}$, the patterns of zeros correspond to Hessenberg functions $(2,3,3,4)$, $(2,3,3,4)$, and $(1,2,3,4)$ respectively. None of these are abelian Hessenberg functions. \par 
On the other hand, the following skew shapes are pre-abelian, and appear with the corresponding Jacobi-Trudi matrices.
\begin{center}
\begin{tabular}{ccccc}
$(4,4,4,4)/(1)$ && $(6,5,4,4)/(2,1) $ && $(2,2,2,2)$\\\\
$\begin{ytableau}
*(light-gray)h_3  & *(light-gray)h_5 & *(light-gray)h_6  & *(light-gray)h_7  \\
*(light-gray)h_2  & *(light-gray)h_4 & *(light-gray)h_5  & *(light-gray)h_6  \\
*(light-gray)h_1  & *(light-gray)h_3 & *(light-gray)h_4  & *(light-gray)h_5  \\
*(pink)1  & *(light-gray)h_2   & *(light-gray)h_3  & *(light-gray)h_4  \\
\end{ytableau}$& &
$\begin{ytableau}
*(light-gray)h_4  & *(light-gray)h_6 & *(light-gray)h_8 & *(light-gray)h_9  \\
*(light-gray)h_2  & *(light-gray)h_4 & *(light-gray)h_6 & *(light-gray)h_7  \\
*(pink)1  & *(light-gray)h_2   & *(light-gray)h_4 & *(light-gray)h_5  \\
0  & *(light-gray)h_1 & *(light-gray)h_3 & *(light-gray)h_4  \\
\end{ytableau}$& &
$\begin{ytableau}
*(light-gray)h_2  & *(light-gray)h_3 & *(light-gray)h_4 & *(light-gray)h_5  \\
*(light-gray)h_1  & *(light-gray)h_2 & *(light-gray)h_3 & *(light-gray)h_4  \\
*(pink)1    & *(light-gray)h_1 & *(light-gray)h_2 & *(light-gray)h_3  \\
0    & *(pink)1   & *(light-gray)h_1 & *(light-gray)h_2  \\
\end{ytableau}$
\end{tabular}
\end{center}
\end{example}
In essence, a skew shape is pre-abelain if the sum (\ref{eqn:expansion}) in Theorem \ref{thm:hook_partition} yields only abelian Hessenberg functions.
\begin{corollary}\label{cor:preabelian}
If $\theta$ is a hook partition and $\mu/\nu$ is pre-abelian, then $\Gamma^\theta_{\mu/\nu}$ is the character of a permutation representation of $S_n$ whose transitive components are each isomorphic to the action of $S_n$ on cosets of a Young subgroup. In other words, under these assumptions Conjecture \ref{conj:Stem_G} holds.
\end{corollary}
\begin{proof}
If $h$ is pre-abelian then $h^J$ is abelian for all $J \subset[n-1]$. So the Hessenberg functions in the decomposition from Corollary \ref{cor:hook_partition} are all abelian. Apply Theorem \ref{thm:abelian}.
\end{proof}
The following result is due to Dahlberg.
\begin{theorem}\label{thm:small}\cite[Thm. 5.4]{Dahlberg19}
If $h$ is such that $h(i) - i \leq 2$, then $\Gamma_h$ is the character of a permutation representation of $S_n$ whose transitive components are each isomorphic to the action of $S_n$ on cosets of a Young subgroup.
\end{theorem}
Note the Dahlberg result is actually stronger, as the paper proves the result for a much larger collection of Hessenberg functions. The larger collection is not as conducive to applying Corollary \ref{cor:hook_partition}.
\begin{corollary}\label{cor:small}
Suppose $\theta$ is a hook partition. If $\mu/\nu$ is a skew shape associated to Hessenberg function $h$ such that $h(i) - i \leq 2$ for all $i \in [n]$, then $\Gamma^\theta_{\mu/\nu}$ is the character of a permutation representation of $S_n$ whose transitive components are each isomorphic to the action of $S_n$ on cosets of a Young subgroup. In other words, under these assumptions Conjecture \ref{conj:Stem_G} holds.
\end{corollary}
\begin{proof}
Each $h^J$ from the decomposition in Corollary \ref{cor:hook_partition} has the property that $h^J(i) \leq h(i)$. Apply Theorem \ref{thm:small}.
\end{proof}
There are several other classes of Hessenberg functions for which the Stanley-Stembridge conjecture is known. Any time the summands from Corollary \ref{cor:hook_partition} are known to fall exclusively within the known cases, we obtain a partial proof of Conjecture \ref{conj:Stem_imm}. In particular, whenever the Jacobi-Trudi matrix does not contain any $1$-s at all, the decomposition in Corollary \ref{cor:hook_partition} will simply be many copies of the original Hessenberg function. This occurs for skew shape $\mu/\nu$ with associated Hessenberg function $h$ when $\mu_{h(i)} -\nu_i + i-h(i) > 0$ for all $i \in [n]$. For any Hessenberg function $h$, it is possible to construct a Jacobi-Trudi matrix whose pattern of nonzero entries corresponds to $h$ and that contains no entries that are $1$. As such, for any particular Hessenberg function $h$ where the Stanley-Stembridge conjecture is known, there are skew shapes $\mu/\nu$ for which the decomposition in equation (\ref{eqn:expansion}) contains $\choos{n-1}{k}$ copies of $\Gamma_h$.

\appendix

\section{Computational Proofs}
A reference for the representation-theoretic calculations below is \cite{sagan2013symmetric}. A reference for the combinatorial calculations for skew-Kostka numbers is \cite[\S 7]{Stanley_Vol2}.\par 
\begin{definition}\label{def:induced}
Let $H$ be a subgroup of $G$. If $\chi$ is a character of $H$, then the \emph{induced character} of $\chi$ on $G$ is 
\[\left.\chi\right\uparrow_{H}^{G}(w) := \frac{1}{\abs{H}}\sum_{x \in G} \chi^\circ(x w x^{-1}) \;\;\text{where}\;\; \chi^\circ(v)  =\begin{cases} \chi(v) & v \in H \\  0 & v \notin H.\end{cases}
\]
for all $w \in G$.
\end{definition}
The following lemma allows us to restrict the $w \in S_n$ we must consider when computing an immanant character whose skew shape is disconnected.
\begin{lemma}\label{lem:disconnected}
Let $\mu/\nu$ be a skew shape and $i \in [n]$ such that $\mu_{i+1} \leq \nu_i$. If $\mu+\delta-w(\nu+\delta)$ has non-negative entries, then $w \in S_{\{1,...,i\}} \times S_{\{i+1,...,n\}}$.
\end{lemma}
\begin{proof}
We proceed by contrapositive. A permutation $w \in S_{\{1,...,i\}} \times S_{\{i+1,...,n\}}$ if and only if $w(\{1,...,i\}) = \{1,...,i\}$. As such, $w \notin S_{\{1,...,i\}} \times S_{\{i+1,...,n\}}$ if and only if there exists a $j \in \{1,...,i\}$ such that $w(j) > i$. Then
\begin{align*}
(\mu+\delta-w(\nu+\delta))_{w(j)} &= \mu_{w(j)}+\delta_{w(j)}-\nu_{j}-\delta_{j} \\
&= (\mu_{w(j)} - \nu_{j}) + (j-w(j))\\
&\leq (\mu_{i+1} - \nu_{i}) + (j-w(j))\\
&\lneq 0.
\end{align*}
This concludes the proof.
\end{proof}

\begin{proposition}[Proposition \ref{prop:no_zeros}]
Say $\mu/\nu$ is a skew shape such that $\mu_i = \nu_i$ for some $i \in [n]$. Let $\widehat{\mu}$ and $\widehat{\nu}$ denote, respectively, the partitions $\mu$ and $\nu$ with their $i$-th components removed. Then ${\ds \Gamma^\theta_{\mu/\nu} = \Gamma^\theta_{\widehat{\mu}/\widehat{\nu}}}$ for all $\theta$.
\end{proposition}
\begin{proof}
Fix $w \in S_n$. We will show that 
\begin{equation}\label{eqn:propA2}
\sum_{w' \in C(w)} K_{\theta,\mu + \delta - w'(\nu + \delta)} = \sum_{w'\in C(w)} K_{\theta,\widehat{\mu} + \delta - w'(\widehat{\nu} + \delta)}.
\end{equation}
Let $w' \in S_n$ such that $\mu+\delta-w'(\nu+\delta)$ has no negative entries. Since $\mu_i = \nu_i$, it follows that $\mu_i \leq \nu_{i-1} $ and $\mu_{i+1} \leq \nu_i$. By Lemma \ref{lem:disconnected}, $w' \in \left(S_{\{1,...,i-1\}} \times S_{\{i,...,n\}}\right)\cap \left(S_{\{1,...,i\}} \times S_{\{i+1,...,n\}}\right)$. In particular, $w' \in S_{\{1,...,i-1\}} \times S_{\{i+1,...,n\}}$. It follows $K_{\mu+\delta-w'(\nu+\delta)} \neq 0$ only if $w' \in S_{\{1,...,i-1\}} \times S_{\{i+1,...,n\}}$.\par 
Now let $w' \in S_n$ such that $\hat{\mu}+\delta-w'(\hat{\nu}+\delta)$ has no negative entries. Since $\hat{\mu}_i \leq \hat{\nu}_{i-1} $ and $0=\hat{\mu}_{n} \leq \hat{\nu}_{n-1}$, by Lemma \ref{lem:disconnected}, $w' \in \left(S_{\{1,...,i-1\}} \times S_{\{i,...,n\}}\right)\cap  S_{\{1,...,n-1\}} $. In particular, $w' \in S_{\{1,...,i-1\}} \times S_{\{i,...,n-1\}}$. It follows $K_{\hat{\mu}+\delta-w'(\hat{\nu}+\delta)} \neq 0$ only if $w' \in S_{\{1,...,i-1\}} \times S_{\{i,...,n-1\}}$.\par 
Consider the automorphism $S_n \to S_n$ given by $v \mapsto c_ivc_i^{-1}$ where $c_i:=(n,n-1,...,i)$ in cycle notation. If $v \in S_{\{1,...,i-1\}}\times S_{\{i+1,...,n\}}$ and $k \in [n]$, then 
\[
c_ivc_i^{-1}(k) = \begin{cases}
v(k) & k \in \{1,...,i-1\} \\
v(k+1) - 1 & k \in \{i,...,n-1\} \\
v(n) = n. & 
\end{cases}
\] 
So $v \mapsto c_ivc_i^{-1}$ is also a bijection from $S_{\{1,...,i-1\}}\times S_{\{i+1,...,n\}} \to S_{\{1,...,i-1\}}\times S_{\{i,...,n-1\}}$, as well as a bijection $C(w) \to C(w)$. We prove equation (\ref{eqn:propA2}) (and thus the claim) by showing for all $w' \in S_n$,
\[
K_{\theta,\mu + \delta - w'(\nu + \delta)} = K_{\theta,\widehat{\mu} + \delta - (c_iw'c_i^{-1})(\widehat{\nu} + \delta)}.
\]
In particular, we will prove that whenever $w' \in S_{\{1,...,i-1\{}\times S_{\{i+1,...,n\}}$,  it is possible to re-order the entries of $\mu + \delta - w'(\nu + \delta)$ to obtain $\widehat{\mu} + \delta - c_iw'c_i^{-1}(\widehat{\nu} + \delta)$. Since it suffices to show this for sequences with non-negative entries, we may assume $w' \in S_{\{1,...,i-1\}}\times  S_{\{i+1,...,n\}}$. \par 
Now $\mu_j = \widehat{\mu}_j$ and $\nu_j=\widehat{\nu}_j$ when $j=1,...,i-1$. Thus for all $j \in \{1,...,i-1\}$,
\begin{align*}
(\mu + \delta)_j - (\nu + \delta)_{(w')^{-1}(j)} &= (\widehat{\mu} + \delta)_j - (\widehat{\nu} + \delta)_{(w')^{-1}(j)}\\
&= (\widehat{\mu} + \delta)_j - (\widehat{\nu} + \delta)_{(c_iw'c_i^{-1})^{-1}(j)}.
\end{align*} 
In other words, the first $i-1$ entries of $\mu + \delta - w'(\nu + \delta)$ and $\widehat{\mu} + \delta - (c_iw'c_i^{-1})(\widehat{\nu} + \delta)$ are identical. \par 
Now $w'(i) = i$, so  $(\mu + \delta - w'(\nu + \delta))_i = 0$ and $c_iw'c_i^{-1}(n) = n$ so $(\widehat{\mu} + \delta - (c_iw'c_i^{-1})(\widehat{\nu} + \delta))_n = 0$ as well. Fix $k \in \{i,...,n-1\}$. Note $\widehat{\mu}_k = \mu_{k+1}$ and $\widehat{\nu}_k = \nu_{k+1}$. We show that the $k$-th element in $\widehat{\mu} + \delta - (c_iw'c_i^{-1})(\widehat{\nu} + \delta)$ is equal to the $k+1$-st element in $\mu + \delta - w'(\nu + \delta)$. 
\begin{align*}
\left(\widehat{\mu} + \delta - (c_iw'c_i^{-1})(\widehat{\nu} + \delta)\right)_k &= \widehat{\mu}_k + \delta_k - (\widehat{\nu} + \delta)_{(c_iw'c_i^{-1})^{-1}(k)} \\
&= \widehat{\mu}_k + \delta_k - \widehat{\nu}_{(w')^{-1}(k+1) - 1} - \delta_{(w')^{-1}(k+1) - 1} \\
&= \mu_{k+1} + \delta_{k+1}+1 - \nu_{(w')^{-1}(k+1)} - \delta_{(w')^{-1}(k+1)}-1\\
&= \left(\mu + \delta - w'(\nu + \delta)\right)_{k+1}.
\end{align*}
Thus we have a bijection between the entries of $\mu + \delta - w'(\nu + \delta)$ and $\widehat{\mu} + \delta - c_iw'c_i^{-1}(\widehat{\nu} + \delta)$, so $K_{\theta,\mu + \delta - w'(\nu + \delta)} = K_{\theta,\widehat{\mu} + \delta - (c_iw'c_i^{-1})(\widehat{\nu} + \delta)}$.
\end{proof}

\begin{proposition}[Proposition \ref{prop:min_n}]
Let $\mu/\nu$ be a skew shape of length at most $n-1$. If $\Gamma_{\mu/\nu}^\theta = \sum_i \Gamma_{\mu_i/\nu_i}^{(N)}$ as characters in $S_{n-1}$, then $\Gamma_{\mu/\nu}^\theta = \sum_i \Gamma_{\mu_i/\nu_i}^{(N)}$ as characters of $S_{n}$. In particular, if Conjecture \ref{conj:main_char} is true for characters $\Gamma^\theta_{\mu/\nu}$ of $S_{\ell(\mu/\nu)}$ then Conjecture \ref{conj:main_char} is true for $\Gamma^\theta_{\mu/\nu}$ characters of $S_m$ where $m \geq \ell(\mu/\nu)$.
\end{proposition}
\begin{proof}
When viewed as a character of $S_n$, denote $\Gamma^\theta_{\mu/\nu}$ as $\Gamma_n$, and by $\Gamma_{n-1}$ when viewed as a character of $S_{n-1}$. For $k\in\{n-1,n\}$, let $C_k(w)$ be the conjugacy classes of $w$ in $S_k$, let $Z_{k}(w)$ be the centralizer of $w$ in $S_k$, and let $\delta^k = (k-1,...,1,0)$. We will show that $\Gamma_n=\left.\Gamma_{n-1}\right\uparrow_{S_{n-1}}^{S_n}$.  
Let $w \in S_n$. If $C_n(w) \cap S_{n-1} = \emptyset$, then $\left.\Gamma_{n-1}\right\uparrow_{S_{n-1}}^{S_n}(w) = 0$ by definition. Since $\ell(\mu/\nu) \leq n-1$, we know that $\mu_n=\nu_n \leq \nu_{n-1}$.  So, if $w' \in S_n$ is such that $\mu+\delta^n-w'(\nu+\delta^n)$ has no negative entry, then $w'(n)=n$ by Lemma \ref{lem:disconnected}.  As $C_n(w)$ conists of derangements, if $w' \in C_n(w)$ then $\mu+\delta^n-w'(\nu+\delta^n)$ has a negative entry. So $\Gamma_n(w) = 0$ as well. \par 
If $C_n(w) \cap S_{n-1} \neq \emptyset$, there exists a $v \in S_{n-1}$ such that $\chi(w)=\chi(v)$ for any class function $\chi$ on $S_n$. Thus it suffices to prove that $\Gamma_n(w) = \left.(\Gamma_{n-1})\right\uparrow_{S_{n-1}}^{S_n}(w)$ for $w \in S_{n-1}$.\par 
Conjugacy classes in symmetric groups are characterized by cycle types, so if $v,w \in S_{n-1}$, then $v \in C_{n-1}(w)$ if and only if $v \in C_n(w)$. In particular, if $w \in S_{n-1}$ then $S_{n-1} \cap C_n(w) = C_{n-1}(w)$. Finally $\abs{Z_n(w)} = \frac{n!}{\abs{C_n(w)}}$, so 
\begin{align*}
\left.(\Gamma_{n-1})\right\uparrow_{S_{n-1}}^{S_n}(w) &= \frac{1}{\abs{S_{n-1}}}\sum_{x \in S_n} \Gamma_{n-1}^\circ(x w x^{-1})\\
&=  \frac{1}{(n-1)!}\abs{Z_n(w)}\sum_{\sigma\in C_{n-1}(w)} \Gamma_{n-1}(\sigma )\\
&= \frac{1}{(n-1)!}\abs{Z_n(w)}\abs{C_{n-1}(w)} \Gamma_{n-1}(w)\\
&=\frac{n!\abs{C_{n-1}(w)}}{(n-1)!\abs{C_n(w)}}  \left(\frac{(n-1)!} {\abs{C_{n-1}(w)}}   \sum_{w'\in C_{n-1}(w)} K_{\theta,\mu + \delta^{n-1} - w'(\nu + \delta^{n-1})}\right)\\
&=\frac{n!}{\abs{C_n(w)}}\sum_{w'\in C_{n-1}(w)} K_{\theta,\mu + \delta^{n-1} - w'(\nu + \delta^{n-1})}.
\end{align*}
On the other hand, the first $n-1$ elements of $\delta^n$ are exactly one greater than the those in $\delta^{n-1}$. More specifically, under point-wise addition of integer sequences (and considering $\delta^{n-1}$ as a sequence of length $n$ by appending the integer $0$), $\delta^{n} = \delta^{n-1} + (1,1,...,1)$. In particular, noting $w'((1,1,...,1)) = (1,1,...,1)$, up to appending a $0$, $\delta^{n-1}-w'(\delta^{n-1}) = \delta^n-w'(\delta^n)$. So
\begin{align*}
\Gamma_n(w) &= \frac{n!}{\left|C_n(w)\right|} \sum_{w'\in C_n(w)} K_{\theta,\mu + \delta^n - w'(\nu + \delta^n)}\\
&=\frac{n!}{\abs{C_n(w)}}\sum_{w'\in C_{n-1}(w)} K_{\theta,\mu + \delta^{n-1} - w'(\nu + \delta^{n-1})}.
\end{align*} 
 Thus ${\ds \left.\Gamma_n(w) = (\Gamma_{n-1})\right\uparrow_{S_{n-1}}^{S_n}(w)}$ for all $w \in S_{n-1}$. \par 
The proposition now follows by linearity of induced characters. The ``particular" part of our claim follows by induction on $n$.
\end{proof}
\begin{proposition}[Proposition \ref{prop:connected_order}]
Let $\mu^0/\nu^0$ and $\mu^1/\nu^1$ be two skew shapes whose Young diagrams have identical connected components. Let $\theta \vdash N =  \abs{\mu^0/\nu^0} = \abs{\mu^1/\nu^1}$. Then ${\ds \Gamma^\theta_{\mu^0/\nu^0} = \Gamma^\theta_{\mu^1/\nu^1}}$.
\end{proposition}
\begin{proof}
This will proceed similarly to the proof of Proposition \ref{prop:no_zeros} (Proposition A.2). First, let $n=\len(\mu^0/\nu^0) = \len(\mu^1/\nu^1)$, so $\Gamma^\theta_{\mu^0/\nu^0}$ and $\Gamma^\theta_{\mu^1/\nu^1}$ are characters of $S_n$.\par
A disconnected skew shape $\mu/\nu$ is naturally associated to a Young subgroup $S$ of $S_n$, defined by the rule that the simple transposition $(i,i+1)\in S$ if $\mu_{i+1} \gneq \nu_i$ (i.e. if the $i$-th and $i+1$-th rows of $\mu/\nu$ are in the same connected component). By Lemma \ref{lem:disconnected}, $K_{\theta,\mu +\delta-w'(\nu+\delta)} = 0$ whenever $w'\notin S$. Let $S^0$ and $S^1$ be the Young subgroups corresponding to $\mu^0/\nu^0$ and $\mu^1/\nu^1$ respectively. Let $\I^0,\I^1$ be, respectively, sets orbits of $S^0$ and $S^1$ on $[n]$. \par  
Finally let $\sigma\in S_n$ so that permuting the rows of $\mu^0/\nu^0$ via $\sigma$ (i.e. sending row $i$ to row $\sigma(i)$) gives $\mu^1/\nu^1$. We require that $\sigma$ be order preserving within the row indices of each connected component of $\mu^0/\nu^0$.\par 
We prove that
\begin{equation}\label{eqn:conn1}
\sum_{w' \in C(w)} K_{\theta,\mu^0 + \delta - w'(\nu^0 + \delta)} = \sum_{w'\in C(w)} K_{\theta,\mu^1 + \delta - w'(\nu^1 + \delta)}.
\end{equation}
Let $\phi_\sigma \colon S_n \to S_n$ be the automorphism $\phi_\sigma(w) = \sigma w \sigma^{-1}$. Note $\phi_\sigma(C(w)) = C(w)$. Equation (\ref{eqn:conn1}) will follow directly once we show that
\[
K_{\theta,\mu^0 + \delta - w'(\nu^0 + \delta)} =  K_{\theta,\mu^1 + \delta - \phi_\sigma(w')(\nu^1 + \delta)}
\] 
for all $w \in S_n$.\par 
We observe that $\phi_\sigma(S^0) = S^1$, and that $\sigma$ is order preserving within the row indices of each connected component of $\mu^0/\nu^0$, so for all $s,t \in I^0 \in \I^0$ and $w' \in S^0$,
\begin{equation*}
\sigma(t)-\sigma(s) = t-s,\;\;\; \text{and }\;\;t-w'(s) = \sigma(t)-\phi_\sigma(w')(\sigma(s)).
\end{equation*}
Let $k \in I^0\in \I^0$ and $w' \in S^0$. Then $w'^{-1}(k) \in I^0$ and
\begin{align*}
\mu^0_{k} - w'(\nu^{0})_k &=\mu^0_{k} - \nu^0_{{w'}^{-1}(k)}\\ 
&=\mu^1_{\sigma(k)}-\nu^1_{\sigma {w'}^{-1}(k)}\\
&= \mu^1_{\sigma(k)}-\nu^1_{\phi_\sigma({w'}^{-1}) \sigma(k)}\\
&= \mu^1_{\sigma(k)}-\phi_\sigma(w') (\nu^{1})_{\sigma(k)}.
\end{align*}
Similarly  
\begin{align*}
\delta_k-w'(\delta)_k &= (n-k-1) - (n-{w'}^{-1}(k)-1) \\
&= {w'}^{-1}(k)-k \\
&= \phi_\sigma({w'}^{-1})(\sigma(k))-\sigma(k)\\
&= \delta_{\sigma(k)} - \phi_\sigma(w')(\delta)_{\sigma(k)}.
\end{align*}
Combining the previous two calculations, it follows that
\begin{align*}
(\mu^0+\delta-w'(\nu^0-\delta))_k &= \mu^0_{k} - w'(\nu^0)_k + \delta_k-w'(\delta)_k\\
&= \mu^1_{\sigma(k)}-\phi_\sigma(w') (\nu^{1})_{\sigma(k)}+\delta_{\sigma(k)} - \phi_\sigma(w')(\delta)_{\sigma(k)}\\
&= (\mu^1+\delta - \phi_\sigma(w')(\nu^1-\delta))_{\sigma(k)}.
\end{align*}
We conclude that $K_{\theta,\mu^0+\delta-w'(\nu^0+\delta)} = K_{\theta,\mu^1+\delta-\phi_\sigma(w')(\nu^1+\delta))}$. 
\end{proof}

\begin{proposition}[Proposition \ref{prop:disconnect}]
Let $\mu/\nu$ be a disconnected skew shape with two component skew shapes $\mu^k/\nu^k$ and $\mu^r/\nu^r$. Let $N_k=\abs{\mu^k/\nu^k}$ and $N_r = \abs{\mu^r/\nu^r}$. 
Then
\[
\Gamma^{\theta}_{\mu/\nu} =\sum_{\substack{ \lambda \vdash  N_k\\\lambda <\theta}} \sum_{\sigma \vdash N_r} c^\theta_{\lambda\sigma}\Gamma^{\lambda}_{\mu^k/\nu^k} \circ \Gamma^{\sigma}_{\mu^r/\nu^r},
\]
where $c^{\theta}_{\lambda\sigma}$ is a Littlewood-Richardson coefficient.
\end{proposition}
\begin{proof}
We will abuse notation and let $k = \ell(\mu^k/\nu^k)$ and $r = \ell(\mu^r/\nu^r)$. First, we require two facts about Kostka numbers. The skew-Kostka number $K_{\theta/\lambda,c}$ for skew shape $\theta/\lambda$ and finite integer sequence $c$ is the number of semi-standard tableaux of shape $\theta/\lambda$ and content $c$. Let $\theta \vdash N$ and $c = (c_1,...,c_k,...,c_{k+r})$ be a finite sequence of non-negative integers that sum to $N$. Let $M_k = \sum_{i=1}^k c_i$. Then
\begin{equation}
K_{\theta,c} = \sum_{\substack{ \lambda\vdash M_k \\\lambda < \theta}} K_{\lambda,(c_1,...,c_k)}\cdot K_{\theta / \lambda, (c_{k+1},...,c_{k+r})}.
\end{equation}
Secondly, skew Kostka numbers are sums of Kostka numbers via the Littlewood-Richardson rule \cite{Stanley_Vol2}. Formally,
\begin{equation}
K_{\theta/\lambda,c} = \sum_{\sigma} c^{\theta}_{\lambda \sigma} K_{\sigma,c}.
\end{equation}
As in the proof of Proposition \ref{prop:min_n} (Proposition A.3), we let $\delta^r = (r-1,...,1,0)$. For any element $x = x_kx_r$ of $S_{k} \times S_{r}$, 
 \[\mu+\delta^n - x(\nu+\delta^n) = (\mu^k + \delta^{n_k} - x_k(\nu^k + \delta^{k}))\cdot (\mu^r + \delta^{r} - x_r(\nu^r + \delta^{r})),\]
 where $\cdot$ is concatenation of sequences.\par 
Let $w \in S_n$ be arbitrary. We proceed by evaluating and simplifying the expression
 \[
 \sum_{\substack{\lambda <\theta \\ \lambda\vdash N_k}} \sum_{\sigma} c^\theta_{\lambda\sigma}\Gamma^{\lambda}_{\mu^k/\nu^k} \circ \Gamma^{\sigma}_{\mu^r/\nu^r}(w).
 \]
 By definition, 
 \begin{align*}
 \Gamma^{\lambda}_{\mu^k/\nu^k} \circ \Gamma^{\sigma}_{\mu^r/\nu^r}(w) &= \left.\left( \Gamma^{\lambda}_{\mu^k/\nu^k} \times \Gamma^{\sigma}_{\mu^r/\nu^r}\right)\right\uparrow_{S_k \times S_r}^{S_n}(w)\\
 &=  \frac{1}{k!r!} \sum_{x \in S_n} \left( \Gamma^{\lambda}_{\mu^k/\nu^k} \times \Gamma^{\sigma}_{\mu^r/\nu^r}\right)^\circ\left(x^{-1}wx\right)\\
 &=  \frac{1}{k!r!} \abs{\Ce_n(w)}\sum_{x \in C_n(w)} \left( \Gamma^{\lambda}_{\mu^k/\nu^k} \times \Gamma^{\sigma}_{\mu^r/\nu^r}\right)^\circ\left(x\right).
 \end{align*}
Now ${\displaystyle \Gamma^{\lambda}_{\mu^k/\nu^k} \times \Gamma^{\sigma}_{\mu^r/\nu^r}}$ is defined to be zero on all elements \emph{not} in $S_k \times S_r$. To shorten notation moving forward, let $C_n^{kr}(w) := C_n(w) \cap \left(S_k\times S_r\right)$. Then 
 \begin{align*}
 \Gamma^{\lambda}_{\mu^k/\nu^k} \circ \Gamma^{\sigma}_{\mu^r/\nu^r}(w) &= \frac{1}{k!r!} \abs{\Ce_n(w)}\sum_{x_kx_r \in C_n^{kr}(w)} \left( \Gamma^{\lambda}_{\mu^k/\nu^k} \times \Gamma^{\sigma}_{\mu^r/\nu^r}\right)\left(x_kx_r\right)\\
 &= \frac{1}{k!r!} \abs{\Ce_n(w)}\sum_{x_kx_r \in C_n^{kr}(w)} \Gamma^{\lambda}_{\mu^k/\nu^k}  \left(x_k\right)\Gamma^{\sigma}_{\mu^r/\nu^r}\left(x_r\right).
 \end{align*}
By definition
\begin{align*}
\Gamma^{\lambda}_{\mu^k/\nu^k}  \left(x_k\right) &= \frac{k!}{\abs{C_k(x_k)}} \sum_{x_k' \in C_k(x_k)} K_{\lambda, \mu^k-\delta^k + x_k'(\nu^k-\delta^k)},\text{ and}\\
\Gamma^{\sigma}_{\mu^r/\nu^r}  \left(x_r\right) &= \frac{r!}{\abs{C_r(x_r)}} \sum_{x_r' \in C_r(x_r)} K_{\sigma, \mu^r-\delta^r + x_r'(\nu^r-\delta^r)}.
\end{align*}
To further shorten notation, let ${\displaystyle K_{\lambda,\widehat{x_k'}}}$ and ${\displaystyle K_{\sigma,\widehat{x_r'}}}$ denote the Kostka numbers in the above sums. Now ${\displaystyle \Gamma^{\lambda}_{\mu^k/\nu^k} \circ \Gamma^{\sigma}_{\mu^r/\nu^r}(w)}$ is equal to 
\begin{align*}
&\frac{\abs{\Ce_n(w)}}{k!r!} \sum_{x_kx_r \in C_n^{kr}(w)} \left(\frac{k!}{\abs{C_k(x_k)}} \sum_{x_k' \in C_k(x_k)} K_{\lambda,\widehat{x_k'}}\right)\left(\frac{r!}{\abs{C_r(x_r)}} \sum_{x_r' \in C_r(x_r)} K_{\sigma, \widehat{x_r'}}\right)\\
&=\abs{\Ce_n(w)}\sum_{x_kx_r \in C_n^{kr}(w)} \frac{1}{\abs{C_k(x_k)}\abs{C_r(x_r)}} \sum_{\substack{x_k' \in C_k(x_k)\\x_r' \in C_r(x_r)}} K_{\lambda,\widehat{x_k'}}  K_{\sigma, \widehat{x_r'}}.
\end{align*}
Now we decompose $C_n^{kr}(w)$ further. If $\lambda_k,\lambda_r$ are partitions, we write $\lambda_k\cdot \lambda_r$ for the partition constructed by concatenating $\lambda_k$ and $\lambda_r$ and reordering to be decreasing. Let $\rho_n(w)$ be the cycle type of $w$ in $S_n$ so that $C_n(w) = \{w' \in S_n \mid \rho_n(w') = \rho_n(w)\}$. Define $\rho_k$ and $\rho_r$ similarly. Let $C_n(\lambda)$ denote the conjugacy class of cycle type $\lambda$ in $S_n$. Then 
\begin{align*}
C_n(w)\cap \left(S_k\times S_r\right) &= \{x_kx_r \in S_k \times S_r \mid \rho_k(x_k)\cdot \rho_r(x_r) = \rho_n(w)\} \\
&= \bigsqcup_{\substack{\lambda_k\cdot \lambda_r = \rho_n(w) \\ \lambda_k \vdash k,\; \lambda_r \vdash r}} \{x_kx_r \in S_k \times S_r \mid \rho_k(x_k)=\lambda_k,\;\; \rho_r(x_r) = \lambda_r)\} \\
&= \bigsqcup_{\substack{\lambda_k\cdot \lambda_r = \rho_n(w) \\ \lambda_k \vdash k,\; \lambda_r \vdash r}} C_k(\lambda_k)\times C_r(\lambda_r)
\end{align*}
Now ${\ds \Gamma^\lambda_{\mu^k/\nu^k} \circ \Gamma^{\sigma}_{\mu^r/\nu^r}(w)}$ simplifies to 
\begin{align*}
&\abs{\Ce_n(w)}\sum_{\substack{\lambda_k\cdot \lambda_r = \rho_n(w) \\ \lambda_k \vdash k,\; \lambda_r \vdash r}} \sum_{\substack{x_k \in C_w(\lambda_k)\\x_r \in C_r(\lambda_r)}} \frac{1}{\abs{C_k(\lambda_k)}\abs{C_r(\lambda_r)}} \sum_{\substack{x_k' \in C_k(\lambda_k)\\x_r' \in C_r(\lambda_r)}} K_{\lambda,\widehat{x_k'}}  K_{\sigma, \widehat{x_r'}}\\
&=\abs{\Ce_n(w)}\sum_{\substack{\lambda_k\cdot \lambda_r = \rho_n(w) \\ \lambda_k \vdash k,\; \lambda_r \vdash r}} \sum_{\substack{x_k' \in C_k(\lambda_k)\\x_r' \in C_r(\lambda_r)}} K_{\lambda,\widehat{x_k'}}  K_{\sigma, \widehat{x_r'}}\\
&=\abs{\Ce_n(w)}\sum_{x_kx_r \in C_n^{kr}(w)}  K_{\lambda,\widehat{x_k}}  K_{\sigma, \widehat{x_r}}.
\end{align*}
Concatenating $\widehat{x_k}$ and $\widehat{x_r}$ gives $\widehat{x_kx_r} = \widehat{x}$. From equations A.2 and A.3, we see that
\begin{align*}
 \sum_{\substack{ \lambda \vdash N_k\\\lambda <\theta}} \sum_{\sigma \vdash N_r} c^\theta_{\lambda\sigma}\Gamma^{\lambda}_{\mu^k/\nu^k} \circ \Gamma^{\sigma}_{\mu^r/\nu^r}(w) &=  \sum_{\substack{\lambda \vdash N_k\\\lambda <\theta}} \sum_{\sigma\vdash N_r} c^\theta_{\lambda\sigma}\abs{\Ce_n(w)}\sum_{x_kx_r \in C_n^{kr}(w)}  K_{\lambda,\widehat{x_k}}  K_{\sigma, \widehat{x_r}} \\
&= \abs{\Ce_n(w)}\sum_{x_kx_r \in C_n^{kr}(w)}  \sum_{\substack{\lambda \vdash N_k\\\lambda <\theta}}  K_{\lambda,\widehat{x_k}}  \sum_{\sigma \vdash N_r} c^\theta_{\lambda\sigma}K_{\sigma, \widehat{x_r}}\\
&= \abs{\Ce_n(w)}\sum_{x_kx_r \in C_n^{kr}(w)}  \sum_{\substack{\lambda \vdash N_k\\\lambda <\theta}}  K_{\lambda,\widehat{x_k}} K_{\theta/\lambda, \widehat{x_r}}\\
&= \abs{\Ce_n(w)}\sum_{x_kx_r \in C_n^{kr}(w)}   K_{\theta,\widehat{x_k}\cdot \widehat{x_r}}\\
&= \abs{\Ce_n(w)}\sum_{x \in C_n^{kr}(w)}   K_{\theta,\widehat{x}}.
\end{align*}
By Lemma \ref{lem:disconnected}, $K_{\theta,\widehat{x}} = 0$ whenever $x \notin S_k\times S_r$, so ${\ds \sum_{x \in C_n^{kr}(w)}   K_{\theta,\widehat{x}} = \sum_{x \in C_n(w)}   K_{\theta,\widehat{x}}}$, and we obtain exactly ${\displaystyle \Gamma^{\theta}_{\mu/\nu}(w)}$. Since $w$ was arbitrary, this proves the proposition.
 \end{proof}

\bibliography{mybib}
\bibliographystyle{alpha}

\end{document}